\documentclass[a4paper,12pt]{article}
\usepackage{amsmath,amsthm,amsfonts,amssymb,bbm,graphics,psfrag,amscd}
\usepackage{fullpage}
\usepackage{epic,eepic}
\title{Asymptotics of the allele frequency spectrum associated with
the Bolthausen-Sznitman coalescent}

\author{Anne-Laure Basdevant \thanks{Laboratoire de Probabilit\'es et
Mod\`eles Al\'eatoires, Universit\'e Pierre et Marie Curie (Paris VI)}
\and Christina Goldschmidt \thanks{Department of Statistics,
University of Oxford}}

\date{}

\newcommand{\E}[1]{\ensuremath{\mathbb{E} \left[#1 \right]}}
\newcommand{\Prob}[1]{\ensuremath{\mathbb{P} \left(#1 \right)}}

\newcommand{\R}{\ensuremath{\mathbb{R}}}
\newcommand{\Z}{\ensuremath{\mathbb{Z}}}
\newcommand{\N}{\ensuremath{\mathbb{N}}}

\newcommand{\ch}[2]{\ensuremath{\left( \begin{smallmatrix} #1 \\ #2
\end{smallmatrix} \right)}}

\newcommand{\convdistn}{\ensuremath{\stackrel{d}{\rightarrow}}}
\newcommand{\convprob}{\ensuremath{\stackrel{p}{\rightarrow}}}

\newcommand{\sumstack}[2]{\ensuremath{\sum_{\substack{#1 \\ #2}}}}

\newcommand{\IO}{\ensuremath{\mathbbm{1}_{\Omega_{n,1}}}}
\newcommand{\deff}{\ensuremath{\overset{\hbox{\tiny{def}}}{=}}}
\newcommand{\convas}{\ensuremath{\overset{\text{a.s.}}{\longrightarrow}}}

\newtheorem{thm}{Theorem}[section]

\newtheorem{lemma}[thm]{Lemma}
\newtheorem{prop}[thm]{Proposition}

\begin{document}
\maketitle

\begin{abstract}
  \noindent We work in the context of the infinitely many alleles
  model.  The allelic partition associated with a coalescent process
  started from $n$ individuals is obtained by placing mutations along
  the skeleton of the coalescent tree; for each individual, we trace
  back to the most recent mutation affecting it and group together
  individuals whose most recent mutations are the same.  The number of
  blocks of each of the different possible sizes in this partition is
  the allele frequency spectrum.  The celebrated Ewens sampling
  formula gives precise probabilities for the allele frequency
  spectrum associated with Kingman's coalescent.  This (and the
  degenerate star-shaped coalescent) are the only
  $\Lambda$-coalescents for which explicit probabilities are known,
  although they are known to satisfy a recursion due to M\"ohle.
  Recently, Berestycki, Berestycki and Schweinsberg have proved
  asymptotic results for the allele frequency spectra of the
  Beta$(2-\alpha,\alpha)$ coalescents with $\alpha \in (1,2)$.  In this
  paper, we prove full asymptotics for the case of the
  Bolthausen-Sznitman coalescent.
\end{abstract}

\section{Introduction}

\subsection{Exchangeable random partitions}

In recent years, the topic of exchangeable random partitions has
received a lot of attention (see Pitman~\cite{PitmanStFl} for a lucid
introduction).  A random partition of $\N$ is said to be
\emph{exchangeable} if, for any permutation $\sigma: \N \to \N$ such
that $\sigma(i) = i$ for all $i$ sufficiently large, we have that the
distribution of the partition is unaffected by the application of
$\sigma$.  It was proved by Kingman~\cite{Kingman2,Kingman1} that if
the partition has blocks $(B_i, i \geq 1)$ listed in increasing order
of least elements then the \emph{asymptotic frequencies},
\[
f_i \deff \lim_{n \to \infty} \frac{|B_i \cap \{1,2,\ldots,n\}|}{n}, i
\geq 1,
\]
exist almost surely.  Let $(f_i^{\downarrow})_{i \geq 1}$ be the
collection of asymptotic frequencies ranked in decreasing order.  Then
we can view $(f_i^{\downarrow})_{i \geq 1}$ as a partition of $[0,1]$
into intervals of decreasing length.  In general, since it is possible
that $\sum_{i \geq 1}f^{\downarrow}_i < 1$, there will also be a
distinguished interval of length $1 - \sum_{i \geq
  1}f^{\downarrow}_i$.  Consider now the following \emph{paintbox
  process}, which creates a random partition of $\N$ starting from the
frequencies.  Take independent uniform random variables $U_1, U_2,
\ldots$ on $[0,1]$.  If $U_i$ and $U_j$ land in the same
non-distinguished interval of the partition then assign $i$ and $j$ to
be in the same block.  If $U_i$ lands in the distinguished interval,
assign $i$ to a singleton block.  The partition we create in this way
is exchangeable and has the same distribution as the partition with
which we began.  This procedure can also be thought of in terms of a
classical balls-in-boxes problem with infinitely many unlabelled
boxes, see in particular Karlin~\cite{Karlin} and Gnedin, Hansen and
Pitman~\cite{Gnedin/Hansen/Pitman}.

There are several natural questions that we may ask about an
exchangeable random partition restricted to the first $n$ integers
(or, equivalently, about the partition formed by the first $n$ uniform
random variables in the paintbox process).  How many blocks does this
partition have?  How many blocks does it have of size exactly $k$, for
$1 \leq k \leq n$?  Even in the absence of precise distributional
information for finite $n$, can we obtain $n \to \infty$ limits for
these quantities, in an appropriate sense?  These questions have been
studied for various classes of exchangeable random partitions and
random compositions, see in particular the work of Gnedin, Pitman and
co-authors: \cite{Barbour/Gnedin, GnedinVienna, Gnedin/Pitman/Yor1,
  Gnedin/Pitman/Yor2, Gnedin:sieve, Gnedin/Yakubovich}.

\subsection{Coalescent process and allelic partitions}

In this paper, we study a particular exchangeable random partition
which derives from a coalescent process.  The origins of this
partition lie in population genetics and we will now describe how it
arises and give a brief review of the relevant literature.  For large
populations, genealogies are often modelled using Kingman's
coalescent~\cite{Kingman}.  This is a Markov process taking values in
the space of partitions of $\N$ (or $[n] \deff \{1,2,\ldots,n\}$),
such that the partition becomes coarser and coarser with time.
Whenever the current state has $b$ blocks, any pair of them coalesces
at rate 1, independently of the other blocks and irrespective of the
block sizes.  We start with a sample of genetic material from $n$
individuals.  Here, $n$ is taken to be small compared to the total
underlying population size.  We imagine tracing the genealogy of the
sample \emph{backwards in time} from the present.  Then the blocks of
the coalescent process at time $t$ correspond to the groups of
individuals having the same ancestor time $t$ ago (where time is
measured in units of the total underlying population size).  See
Ewens~\cite{Ewens} or Durrett~\cite{Durrett} for full introductions to
this subject.  In the population genetics setting, it is natural to
introduce the concept of mutation into this model.  One of the most
celebrated results in this area is the \emph{Ewens Sampling Formula},
which was proved by Ewens~\cite{Ewens72} in 1972.  It concerns the
\emph{infinitely many alleles model}, in which every mutation gives
rise to a completely new type.  It says that if we take a sample of
$n$ genes subject to neutral mutation (that is, mutation which does
not confer a selective advantage) which occurs at rate $\theta/2$ for
each individual, then the probability $q(m_1, m_2, \ldots)$ that there
are $m_j$ types which occur exactly $j$ times is given by
\[
q(m_1, m_2, \ldots) = \frac{n! \theta^{\sum_{i \geq 1} m_i}}{(\theta)_{n
    \uparrow} \prod_{j \geq 1} j^{m_j} m_j!},
\]
where $(\theta)_{n \uparrow} = \theta (\theta + 1) \cdots (\theta +
n-1)$ and we must have $\sum_{j \geq 1} j m_j = n$.  Another way of
expressing this (due to Kingman~\cite{Kingman2}) is to picture the
coalescent tree associated with Kingman's coalescent and place
mutations along the length of the skeleton as a Poisson process of
intensity $\theta/2$.  For each individual, trace backwards in time
(i.e.\ forwards in coalescent time) to the most recent mutation.
Group together those individuals whose most recent mutations are the
same; this gives the \emph{allelic partition}.  Then $m_j$ is the
number of blocks in the allelic partition containing exactly $j$
individuals.  

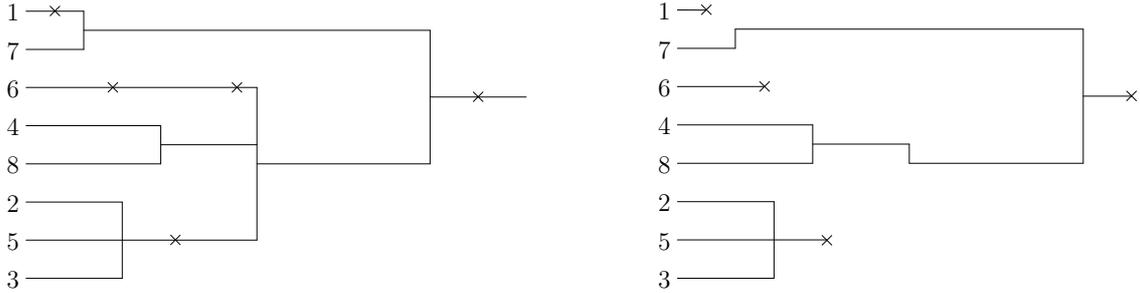
\begin{figure} 
\begin{center}
\resizebox{2.86875in}{!}{\setlength{\unitlength}{0.00083333in}
\begingroup\makeatletter\ifx\SetFigFont\undefined%
\gdef\SetFigFont#1#2#3#4#5{%
  \reset@font\fontsize{#1}{#2pt}%
  \fontfamily{#3}\fontseries{#4}\fontshape{#5}%
  \selectfont}%
\fi\endgroup%
{\renewcommand{\dashlinestretch}{30}
\begin{picture}(4062,2250)(0,-10)
\path(337,2214)(412,2139)
\path(337,2139)(412,2214)
\path(3637,1540)(3712,1465)
\path(3637,1465)(3712,1540)
\path(789,1614)(864,1539)
\path(789,1539)(864,1614)
\path(1757,1614)(1832,1539)
\path(1757,1539)(1832,1614)
\path(1277,415)(1352,340)
\path(1277,340)(1352,415)
\path(150,1875)(600,1875)
\path(600,2175)(600,1875)
\path(600,2025)(3300,2025)
\path(150,1575)(1950,1575)
\path(150,1275)(1200,1275)
\path(150,975)(1200,975)
\path(1200,1275)(1200,975)
\path(150,675)(900,675)
\path(150,375)(1950,375)
\path(150,75)(900,75)
\path(900,675)(900,75)
\path(1200,1125)(1950,1125)
\path(1950,1575)(1950,375)
\path(1950,975)(3300,975)
\path(3300,2025)(3300,975)
\path(3300,1500)(4050,1500)
\path(150,2175)(600,2175)
\put(0,1500){\makebox(0,0)[lb]{{\SetFigFont{12}{14.4}{\familydefault}{\mddefault}{\updefault}6}}}
\put(0,1800){\makebox(0,0)[lb]{{\SetFigFont{12}{14.4}{\familydefault}{\mddefault}{\updefault}7}}}
\put(0,2100){\makebox(0,0)[lb]{{\SetFigFont{12}{14.4}{\familydefault}{\mddefault}{\updefault}1}}}
\put(0,1200){\makebox(0,0)[lb]{{\SetFigFont{12}{14.4}{\familydefault}{\mddefault}{\updefault}4}}}
\put(0,900){\makebox(0,0)[lb]{{\SetFigFont{12}{14.4}{\familydefault}{\mddefault}{\updefault}8}}}
\put(0,600){\makebox(0,0)[lb]{{\SetFigFont{12}{14.4}{\familydefault}{\mddefault}{\updefault}2}}}
\put(0,300){\makebox(0,0)[lb]{{\SetFigFont{12}{14.4}{\familydefault}{\mddefault}{\updefault}5}}}
\put(0,0){\makebox(0,0)[lb]{{\SetFigFont{12}{14.4}{\familydefault}{\mddefault}{\updefault}3}}}
\end{picture}
}}
\hspace{1cm}
\resizebox{2.65625in}{!}{\setlength{\unitlength}{0.00083333in}
\begingroup\makeatletter\ifx\SetFigFont\undefined%
\gdef\SetFigFont#1#2#3#4#5{%
  \reset@font\fontsize{#1}{#2pt}%
  \fontfamily{#3}\fontseries{#4}\fontshape{#5}%
  \selectfont}%
\fi\endgroup%
{\renewcommand{\dashlinestretch}{30}
\begin{picture}(3724,2250)(0,-10)
\path(337,2214)(412,2139)
\path(337,2139)(412,2214)
\path(3637,1540)(3712,1465)
\path(3637,1465)(3712,1540)
\path(789,1614)(864,1539)
\path(789,1539)(864,1614)
\path(1273,410)(1348,335)
\path(1273,335)(1348,410)
\path(150,1875)(600,1875)
\path(600,2025)(600,1875)
\path(600,2025)(3300,2025)
\path(150,1575)(825,1575)
\path(150,1275)(1200,1275)
\path(150,975)(1200,975)
\path(1200,1275)(1200,975)
\path(150,675)(900,675)
\path(150,75)(900,75)
\path(900,675)(900,75)
\path(1200,1125)(1950,1125)
\path(1950,1125)(1950,975)
\path(1950,975)(3300,975)
\path(3300,2025)(3300,975)
\path(3300,1500)(3675,1500)
\path(150,2175)(373,2174)
\path(150,375)(1308,375)
\put(0,2100){\makebox(0,0)[lb]{{\SetFigFont{12}{14.4}{\familydefault}{\mddefault}{\updefault}1}}}
\put(0,1800){\makebox(0,0)[lb]{{\SetFigFont{12}{14.4}{\familydefault}{\mddefault}{\updefault}7}}}
\put(0,1500){\makebox(0,0)[lb]{{\SetFigFont{12}{14.4}{\familydefault}{\mddefault}{\updefault}6}}}
\put(0,1200){\makebox(0,0)[lb]{{\SetFigFont{12}{14.4}{\familydefault}{\mddefault}{\updefault}4}}}
\put(0,900){\makebox(0,0)[lb]{{\SetFigFont{12}{14.4}{\familydefault}{\mddefault}{\updefault}8}}}
\put(0,600){\makebox(0,0)[lb]{{\SetFigFont{12}{14.4}{\familydefault}{\mddefault}{\updefault}2}}}
\put(0,300){\makebox(0,0)[lb]{{\SetFigFont{12}{14.4}{\familydefault}{\mddefault}{\updefault}5}}}
\put(0,0){\makebox(0,0)[lb]{{\SetFigFont{12}{14.4}{\familydefault}{\mddefault}{\updefault}3}}}
\end{picture}
}}
\end{center}
\caption{Left: a coalescent tree with mutations.  Right: the sections
  of the tree relevant for the formation of the allelic partition.
  Note that from each individual we look back only to the last
  mutation, so that the second mutation on the lineage of 6 is
  ignored.  The allelic partition here is $\{1\}, \{2,3,5\},
  \{4,7,8\}, \{6\}$.  If $N_k(n)$ is the number of blocks of size $k$
  when we start with $n$ individuals, then we have $N_1(8) = 2$,
  $N_2(8) = 0$, $N_3(8) = 2$, $N_3(8) = N_4(8) = \cdots = N_8(8) =
  0$.}
\label{fig:tree}
\end{figure}

It is natural to extend these ideas to more general coalescent
processes.  See Figure~\ref{fig:tree} for an example of a general
coalescent tree and its allelic partition.  The $\Lambda$-coalescents
are a class of Markovian coalescent processes which were introduced by
Pitman~\cite{PitmanLambdaCoal} and Sagitov~\cite{Sagitov}.  Like
Kingman's coalescent, they take as their state-space the set of
partitions of $[n]$ (or, indeed, of the whole set of natural numbers).
Their evolution is such that only one block is formed in any
coalescence event and rates of coalescence depend only on the number
of blocks present and not on their sizes.  Take $\Lambda$ to be a
finite measure on $[0,1]$.  In order to give a formal description of
the coalescent, it is sufficient to give its jump rates.  Whenever
there are $b$ blocks present, any particular $k$ of them coalesce at
rate
\[
\lambda_{b,k} \deff \int_0^1 x^{k-2}(1-x)^{b-k} \Lambda(dx), \quad 2 \leq
k \leq b.
\]
Note that, in contrast to Kingman's coalescent, here we allow
\emph{multiple collisions}; that is, we allow more than two blocks to
join together.  Kingman's coalescent is the case $\Lambda(dx) =
\delta_{0}(dx)$, where unit mass is placed at 0.  The case
$\Lambda(dx) = dx$, called the Bolthausen-Sznitman coalescent, was
introduced by Bolthausen and Sznitman~\cite{Bolthausen/Sznitman} in
the context of spin glasses.  It has many nice properties and appears
to be more tractable than most $\Lambda$-coalescents.  For example,
its marginal distributions are known explicitly
\cite{PitmanLambdaCoal}.  It has been studied in some detail: see, for
example, Pitman~\cite{PitmanLambdaCoal}, Bertoin and Le
Gall~\cite{Bertoin/LeGall}, Basdevant~\cite{Basdevant} and Goldschmidt
and Martin~\cite{Goldschmidt/Martin}.

Another subset of the $\Lambda$-coalescents which has recently been
particularly studied is the \emph{Beta coalescents}, so-called because
$\Lambda$ here is a beta density:
\[
\Lambda(dx) = \frac{1}{\Gamma(2-\alpha)\Gamma(\alpha)} x^{1-\alpha}
(1-x)^{\alpha - 1} dx,
\]
for some $\alpha \in (0,2)$.  (The $\alpha = 1$ case is the
Bolthausen-Sznitman coalescent and, in some sense, $\alpha=2$
corresponds to Kingman's coalescent.)  See Birkner et
al~\cite{Magnificent7} for a representation in terms of
continuous-state branching processes when $\alpha \in (0,2)$.

If we suppose that instead of Kingman's coalescent, the genealogy of
the population evolves according to a general $\Lambda$-coalescent
then, except in the special case of the degenerate star-shaped
coalescent (where $\Lambda(dx) = \delta_1(dx)$), there is no known
explicit expression for the probability $q(m_1, m_2, \ldots)$ of
having $m_j$ blocks in the allelic partition of size $j$.  However,
M\"ohle~\cite{Moehle} has shown that the probabilities $q$ must
satisfy the following recursion:
\[
q(m) = \frac{n \rho}{\lambda_n + n \rho}
q(m - e_1) + \sum_{i=1}^{n-1} \frac{\ch{n}{i+1}
  \lambda_{n,i+1}}{\lambda_n + n \rho} \sum_{j=1}^{n-i}
\frac{j(m_j+1)}{n-i} q(m + e_j - e_{i+j}),
\]
where $\lambda_n = \sum_{k=2}^n \ch{n}{k} \lambda_{n,k}$, $\rho =
\theta/2$, $m = (m_1, m_2, \ldots)$ and $e_i$ is the vector with a 1
in the $i$th co-ordinate and 0 in all the rest.  He has also shown
\cite{Moehle3} that, except in the cases of the star-shaped coalescent
and Kingman's coalescent, the allelic partition is not
\emph{regenerative} in the sense of Gnedin and
Pitman~\cite{Gnedin/Pitman}.  Dong, Gnedin and
Pitman~\cite{Dong/Gnedin/Pitman} have studied various properties of
the allelic partition of a general $\Lambda$-coalescent.  In
particular, they view the allelic partition as the final partition of
a coalescent process with \emph{freeze} (see
Section~\ref{sec:fluidlimit} where we use this formalism) and also
give an alternative description of $q$ as the stationary distribution
of a certain discrete-time Markov chain.

Consider again the Beta coalescents.  Suppose that we start the
coalescent process from the partition of $[n]$ into singletons.  Let
$N_k(n)$ be the number of blocks of size $k$, for $k \geq 1$, and let
$N(n)$ be the total number of blocks, so that $N(n) = \sum_{k=1}^{n}
N_k(n)$.  Then the complete allele frequency spectrum is the vector
\[
(N_1(n), N_2(n), N_3(n), \ldots).
\]
In the case of $\alpha \in (1,2)$, Berestycki, Berestycki and
Schweinsberg~\cite{BBS2,BBS1} have proved that
\begin{gather*}
n^{\alpha-2} N(n) \convprob \frac{\rho \alpha(\alpha - 1)
  \Gamma(\alpha)}{2-\alpha} \\
\intertext{and, for $k \geq 1$, that}
n^{\alpha-2} N_k(n) \convprob \frac{\rho \alpha (\alpha - 1)^2
  \Gamma(k+\alpha -2)}{k!},
\end{gather*}
as $n \to \infty$.

The corresponding convergence results for Kingman's coalescent can be
derived from the Ewens sampling formula: without rescaling, we have
\[
(N_1(n), N_2(n), \ldots) \convdistn (Z_1, Z_2, \ldots),
\]
where $Z_1, Z_2, \ldots$ are independent Poisson random variables such
that $Z_i$ has mean $1/i$.  It follows that
\[
\frac{N(n)}{\log n} \convas 1,
\]
as $n \to \infty$ and, moreover, that
\[
\frac{N(n) - \log n}{\sqrt{\log n}} \convdistn \mathrm{N}(0,1).
\]
It is clear that the Beta coalescents belong to a completely different
asymptotic regime.  

A related problem concerns the \emph{infinitely many sites} model.
Here, as before, we put mutations on the coalescent tree, but this
time we imagine that we trace the genealogy of long stretches of
chromosome from each of our $n$ individuals.  Each time a mutation
arrives, it affects a different site on the chromosome.  The number of
\emph{segregating sites} is the number of sites at which there exists
more than one allele in our sample of chromosomes.  This is simply the
number of mutations on the skeleton of the coalescent tree.  Let
$S(n)$ be the number of segregating sites when we start with a sample
of $n$ individuals.  Clearly the distributions of $S(n)$ and $N(n)$
are related, in that in both cases we count mutations along the
skeleton of the coalescent tree; for $N(n)$, we discard any mutation
which arises on a lineage all of whose members have already mutated.
In \cite{Moehle2}, M\"ohle has studied the limiting distribution of
$S(n)$ in the special case where the measure $x^{-1}\Lambda(dx)$ is
finite (which includes the Beta coalescents with $\alpha \in (0,1)$).
He proves that
\begin{equation} \label{eqn:Moehle}
\frac{S(n)}{n} \convdistn \rho \int_0^{\infty} \exp(-\sigma_t) dt,
\end{equation}
where $(\sigma_t)_{t \geq 0}$ is a drift-free subordinator with L\'evy
measure given by the image under the transformation $x \mapsto
-\log(1-x)$ of the measure $x^{-2} \Lambda(dx)$.

The number of segregating sites is, in turn, closely related to the
\emph{length} of the coalescent tree (i.e.\ the sum of the lengths of
all of the branches) and to the total number of collisions before
absorption.  This has been studied for various 
$\Lambda$-coalescents in \cite{Delmas/Dhersin/S-J,
  Drmota/Iksanov/Moehle/Roesler, Gnedin/Yakubovich2, Iksanov/Moehle}.

\subsection{The Bolthausen-Sznitman allelic partition}

Turning now to the Bolthausen-Sznitman coalescent, Drmota, Iksanov,
M\"ohle and R\"osler \cite{Drmota/Iksanov/Moehle/Roesler} have proved
that
\[
\frac{\log n}{n} S(n) \convprob \rho,
\]
where $S(n)$ is the number of segregating sites.  They have also
proved the corresponding ``central limit theorem'',
\[
\frac{S(n) - \rho a_n}{\rho b_n} \convdistn S,
\]
where $a_n = \frac{n}{\log n} + \frac{n \log \log n}{\log^2 n}$, $b_n
= \frac{n}{\log^2 n}$ and $S$ is a stable random variable having
characteristic function
\[
\exp\left(-\frac{1}{2} \pi |t| + it \log t\right).
\]

The purpose of this paper is to prove the following theorem concerning
the complete allele frequency spectrum of the Bolthausen-Sznitman
coalescent.

\begin{thm} \label{thm:main}
For $k \geq 1$, let $N_k(n)$ be the number of blocks of the allelic
partition of size $k$ when we start with $n$ singleton blocks.  Then
\begin{align*}
\frac{\log n}{n} N_1(n) & \convprob \rho \\
\intertext{and, for $k \geq 2$,}
\frac{(\log n)^2}{n} N_k(n) & \convprob \frac{\rho}{k(k-1)}.
\end{align*}
\end{thm}

As a corollary, we obtain that $N(n)$, which is \emph{a priori}
smaller than $S(n)$, has the same first-order asymptotics:
\[
\frac{\log n}{n} N(n) \convprob \rho.
\]

Suppose that we start a general $\Lambda$-coalescent
$(\Pi(t))_{t \geq 0}$ from the partition of $\N$ into singletons.
Then it has been proved by Pitman~\cite{PitmanLambdaCoal} that either $\Pi(t)$
has only finitely many blocks for all $t > 0$ ($(\Pi(t))_{t \geq 0}$
\emph{comes down from infinity}) or $\Pi(t)$ has infinitely many
blocks for all time ($(\Pi(t))_{t \geq 0}$ \emph{stays infinite}).
See Schweinsberg~\cite{Schweinsberg} for an explicit criterion for
when a $\Lambda$-coalcescent comes down from infinity, in terms of the
$\lambda_{b,k}$'s.  The fundamental difference between the Beta
coalescents for $\alpha \in (1,2)$ and $\alpha \in (0,1]$ (including
the Bolthausen-Sznitman coalescent) is that the former coalescents
come down from infinity and the latter do not.  This accounts for the
fact that in Berestycki, Berestycki and Schweinsberg's result, the
scalings are the same for all different sizes of block as $n$ becomes
large, whereas in our theorem, the singletons must be scaled
differently.  Essentially, coalescence occurs rather slowly and the
overwhelming first-order effect is mutation, which causes the allelic
partition to consist mostly of singletons.  However, at the second
order (i.e.\ considering $(N_2(n), N_3(n), \ldots)$), we can feel the
effect of the coalescence.

We do not claim that our results are of any application in population
genetics: to the best of our knowledge, the Bolthausen-Sznitman
coalescent has not been used to model the genealogy of any biological
population.  Nonetheless, our method may extend to the case of
coalescents which are more biologically realistic.

Our method of proof is of some interest in itself.  We track the
formation of the allelic partition using a certain Markov process, for
which we then prove a fluid limit (functional law of large numbers).
The terminal value of our process gives the allele frequency spectrum
and the fluid limit result, after a little extra work, allows us to
read off the asymptotics.

Fluid limits have been widely used in the analysis of stochastic
networks (see, for example, \cite{Chen/Yao}, \cite{Whitt}) and in the
study of random graphs (\cite{Darling/Norris:hypergraph},
\cite{Pittel/Spencer/Wormald}, \cite{Wormald}).  In some sense, the
prototypical result of the type in which we are interested is the
following: suppose we take a Poisson process, $(X(t))_{t \geq 0}$ of
rate $1$, started from 0.  Then the re-scaled process $(N^{-1}
X(Nt))_{t \geq 0}$ stays close (in a rather strong sense) to the
deterministic function $x(t) = t$, at least on compact time-intervals.
For a general pure jump Markov process, the fluid limit is determined
as the solution to a differential equation.  In this article we have
relied on the neat formulation in Darling and
Norris~\cite{Darling/Norris}.  However, our fluid limit is somewhat
unusual.  Firstly, instead of scaling time up, we actually scale it
down, by a factor of $\log n$.  Moreover, we have three different
``space'' scalings for different co-ordinates of our
(multidimensional) process.

\section{Fluid limit} \label{sec:fluidlimit}

Consider the formation of the allelic partition, starting from the
partition into singletons and run until every individual has
received a mutation.  The easiest way to think of this is to use the
terminology of Dong, Gnedin and Pitman~\cite{Dong/Gnedin/Pitman} in
which blocks have two possible states: \emph{active} and
\emph{frozen}. We start with all blocks active and equal to
singletons. Active blocks coalesce according to the rules of the
Bolthausen-Sznitman coalescent: if there are $b$ active blocks
present then any particular $k$ of them coalesce at rate
$\frac{(k-2)!(b-k)!}{(b-1)!}$. Moreover, every active block becomes
frozen at rate $\rho$ and stays frozen forever (this act of freezing
creates a block in the allelic partition).

The data we will track are as follows.  Let $X_k^n(t)$ be the number
of active blocks of the \emph{coalescent} partition at time $t$
containing $k$ individuals, $k \geq 1$, where we start at time $0$
with $n$ active individuals in singleton blocks.  For $k \geq 1$, let
$Z_k^n(t)$ be the number of blocks of the \emph{allelic} partition of
size $k$ which have already been formed by time $t$ (this is the
number of times so far that an active block containing precisely $k$
individuals has become frozen).  For $d \geq 1$, let $Y_{d+1}^n(t) =
\sum_{k=d+1}^{\infty} X_k^n(t)$, the number of active blocks
containing at least $d+1$ individuals.

It is straightforward to see that, for any $d \geq 1$,
\[
X^{n,d}(t)\deff(X_1^n(t), X_2^n(t), \ldots, X_{d}^n(t),
Y_{d+1}^{n}(t), Z_d^n(t))_{t \geq 0}
\]
is a (time-homogeneous) Markov jump process taking values in $\{0,1,2,
\ldots, n\}^{d+2}$, with
\[
X_1^n(0) = n, \quad X_k^n(0) = 0, \quad 2 \leq k \leq d, \quad
Y_{d+1}^n(0) = 0, \quad Z_d^n(0) = 0.
\]

Now put
\begin{alignat*}{3}
\bar{X}_1^n(t)
& = \frac{1}{n} X_1^n\left(\frac{t}{\log n}\right), \quad
& \bar{X}_k^n(t)
& = \frac{\log n}{n} X_k^n\left(\frac{t}{\log n}\right)
\quad \text{for $k \geq 2$,} \\
\bar{Z}_1^n(t)
& = \frac{\log n}{n} Z_1^n\left(\frac{t}{\log n}\right), \quad
& \bar{Z}_k^n(t)
& = \frac{(\log n)^2}{n} Z_k^n\left(\frac{t}{\log n}\right)
\quad \text{for $k \geq 2$}
\end{alignat*}
and
\[
\bar{Y}_{d+1}^n(t) = \frac{\log n}{n} Y_{d+1}^n\left(\frac{t}{\log n}\right)
\quad \text{for $d \geq 1$}.
\]

Fix $d \geq 1$ and write
$$
 \bar{X}^{n,d}(t) =(\bar{X}_1^n(t), \bar{X}_2^n(t),
\ldots, \bar{X}_d^n(t), \bar{Y}_{d+1}^n(t), \bar{Z}_{d}^n(t))
$$

and define a stopping time
\[
T_n = \inf\{t \geq 0: X^{n,d}(t) = \mathbf{0}\}.
\]
(Note that $T_n$ is the same regardless of the value of $d$.)

For $t \geq 0$, let
\begin{alignat*}{3}
x_1(t) & = e^{-t}, \quad &
x_k(t) & = \frac{te^{-t}}{k(k-1)}, \quad 2 \leq k \leq d, \\
z_1(t) & = \rho (1 - e^{-t}), \quad &
z_k(t) & = \frac{\rho}{k(k-1)} (1 - e^{-t} - t e^{-t}),
\quad 2 \leq k \leq d
\end{alignat*}
and
\[
y_{d+1}(t) = \frac{te^{-t}}{d}.
\]
\[
x^{(d)}(t) = (x_1(t), x_2(t), \ldots, x_d(t), y_{d+1}(t), z_d(t)).
\]
We write $\|\cdot\|$ for the Euclidean norm on $\R^{d+2}$.

\begin{prop} \label{prop:fluidlimit}
Fix $d \geq 1$ and let $t_0 < \infty$.  Then, given $\epsilon > 0$,
\[
\Prob{\sup_{0 \leq t \leq t_0} \| \bar{X}^{n,d}(t) - x^{(d)}(t) \| >
\epsilon} \to 0
\]
as $n \to \infty$.
\end{prop}

This is the key to the following result.

\begin{prop} \label{prop:stoptime}
Take $\delta > 0$.  Then
\[
\Prob{\left|\frac{\log n}{n} Z_1^n(T_n) - \rho\right| > \delta} \to 0
\]
and, for $k \geq 2$,
\[
\Prob{\left|\frac{(\log n)^2}{n} Z_k^n(T_n) - \frac{\rho}{k(k-1)}
  \right| > \delta} \to 0,
\]
as $n \to \infty$.
\end{prop}

Theorem~\ref{thm:main} now follows directly, since $N_k(n) =
Z_k^n(T_n)$ for $k \geq 1$.  Note that
Proposition~\ref{prop:fluidlimit} tells us \emph{how} the allele
frequency spectrum is formed.

\textbf{Remark.}  Delmas, Dhersin and
Siri-Jegousse~\cite{Delmas/Dhersin/S-J} have recently considered the
lengths of coalescent trees associated with Beta coalescents for
$\alpha \in (1,2)$.  Part (1) of their Theorem 5.1 appears to be a
result analogous to our Proposition~\ref{prop:fluidlimit}.

\section{Proofs}

In this section, we prove Proposition~\ref{prop:fluidlimit} and
deduce Proposition~\ref{prop:stoptime}. In order to do so, we use
the fluid limit methodology described in Darling and
Norris~\cite{Darling/Norris}. Firstly, we need to set up some
notation. Let $\beta^{n,d}(m)$ be the drift of the process
$X^{n,d}$ when it is in state $m = (m_1, m_2, \ldots, m_{d+2}) \in
\{0,...,n\}^{d+2}$, so that 
\[
\beta^{n,d}(m) = \sum_{m' \neq m} (m' - m) q^{n,d}(m, m'),
\]
where $q^{n,d}(m, m')$ is the jump rate from $m$ to $m'$. Let
$\alpha^{n,d}(m)$ be the corresponding variance of a jump, in the
sense that
\[
\alpha^{n,d}(m) = \sum_{m' \neq m} \| m' - m \|^2 q^{n,d}(m, m').
\]
Let us also introduce the notation
\[
\alpha^{n,d}_k(m) = \sum_{m' \neq m} | m'_k - m_k |^2 q^{n,d}(m,
m'),
\]
for $1 \leq k \leq d+2$, so that we may decompose $\alpha^{n,d}(m)$ as
\[
\alpha^{n,d}(m) = \sum_{k=1}^{d+2} \alpha^{n,d}_k(m).
\]
Finally, let $M \deff \sum_{k=1}^{d+1} m_k$ denote the total number of
active blocks in the partition. We will need to compute the drift and
infinitesimal variance of the re-scaled process $\bar{X}^{n,d}$, which
takes values in the set
\[
\mathcal{S}^{n,d}\!\deff\! \left\{0,\frac{1}{n},\ldots,1\right\}
\!\times \!\left\{0,\frac{\log n}{n}, 2 \frac{\log n}{n}, \ldots,
\log n\right\}^d \!\!\times \!\left\{0, \frac{(\log n)^{r}}{n}, 2
\frac{(\log n)^{r}}{n}, \ldots, (\log n)^{r}\right\},
\]
where $r=1$ if $d=1$ and $r=2$ if $d\ge 2$.  Denote by
$\bar{\beta}^{n,d}(\xi)$ and $\bar{\alpha}^{n,d}(\xi)$ the drift and
infinitesimal variance of $\bar{X}^{n,d}$ when it is in the state $\xi
= (\xi_1, \xi_2, \ldots, \xi_{d+2}) \in \mathcal{S}^{n,d}$.  Then,
letting $m = (n\xi_1,\frac{n}{\log n}\xi_2,\ldots,\frac{n}{\log
  n}\xi_{d+1},\frac{n}{(\log n)^r}\xi_{d+2})$, we have
\begin{align}
\bar{\beta}_k^{n,d}(\xi) 
&= \begin{cases}
\frac{1}{n\log n}\beta_1^{n,d}(m) & k= 1\\
\frac{1}{n}\beta_k^{n,d}(m) & 2 \leq k \leq d+1 \\
\frac{(\log n)^{r-1}}{n}\beta_{d+2}^{n,d}(m) & k=d+2, \\
\end{cases} \label{eqn:betabar} \\
\bar{\alpha}^{n,d}_k(\xi)
&= \begin{cases}
\frac{1}{n^2\log n} \alpha_1^{n,d}(m) & k = 1 \\
\frac{\log n}{n^2} \alpha_k^{n,d}(m) & 2 \leq k \leq d+1 \\
\frac{(\log n)^{2r-1}}{n^2} \alpha_{d+2}^{n,d}(m) & k = d+2
\end{cases} \notag
\end{align}
and
\[
\bar{\alpha}^{n,d}(\xi) = \sum_{k=1}^{d+1} \bar{\alpha}^{n,d}_k(\xi).
\]
Now define $b^{(d)}: \R^{d+2} \to \R^{d+2}$ co-ordinatewise by
\[
b^{(d)}_k(\xi) = \begin{cases}
-\xi_1 & k = 1\\
\frac{1}{k(k-1)}\xi_1 - \xi_k & 2 \leq k \leq d \\
\frac{1}{d}\xi_1 - \xi_{d+1} & k=d+1 \\
\rho \xi_d & k=d+2.
\end{cases}
\]
Then the vector field $b^{(d)}$ is Lipschitz in the Euclidean norm
with constant $K \deff \sqrt{\rho^2 + \frac{\pi^2}{3}}$.  The function
$x^{(d)}(t)$ of the previous section is the unique solution of the
differential equation
\[
\frac{d}{dt} x^{(d)}(t) = b^{(d)}(x^{(d)}(t)).
\]

In order to prove Proposition~\ref{prop:fluidlimit}, we need a few
lemmas.  Firstly, we prove two analytic results.  For $n \in \N$, let
$h(n) = \sum_{i=1}^{n-1} \frac{1}{i}$, the $(n-1)$th harmonic number.

\begin{lemma} \label{lem:htolog}
Fix $R > e$.  Then for $x \in \frac{1}{n} \Z \cap [R^{-1},1]$,
\[
\left|\frac{h(nx)}{\log n} - 1 \right| \leq \frac{\log R}{\log n}.
\]
\end{lemma}

\begin{proof}
It is an elementary fact that, for $k \geq 2$,
\[
\log(k) \leq h(k) \leq 1 + \log(k-1) \leq 1 + \log(k).
\]
This entails that
\[
\left| \frac{h(nx)}{\log n} - 1 \right| \leq
\max \left\{-\frac{\log(x)}{\log n}, \frac{1 + \log(x)}{\log n}
\right\} \leq \frac{\log R}{\log n}
\]
in the specified range of $x$.
\end{proof}

\begin{lemma} \label{lem:binomcoeff}
For $0 \leq j \leq n$ and $k \geq 0$,
\[
0 \leq 1 - \frac{\ch{n}{j}}{\ch{n+k}{j}} \leq \frac{kj}{n-j+1}.
\]
\end{lemma}

\begin{proof}
We have
\[
\log \left(\frac{\ch{n}{j}}{\ch{n+k}{j}}\right)
= - \sum_{i=0}^{j-1} ( \log(n - i + k) - \log(n - i) ).
\]
By the mean value theorem,
\[
\log(n - i + k) - \log(n - i) \leq \frac{k}{n-i}, \quad 0 \leq i \leq n-1.
\]
Hence,
\[
\sum_{i=0}^{j-1} ( \log(n - i + k) - \log(n - i) )
\leq \sum_{i=0}^{j-1} \frac{k}{n-i} \leq \frac{kj}{n-j+1}
\]
and so
\[
\log \left(\frac{\ch{n}{j}}{\ch{n+k}{j}}\right) \geq - \frac{kj}{n-j+1}.
\]
Since
\[
\exp \left(-\frac{kj}{n-j+1}\right)
\geq 1 - \frac{kj}{n-j+1},
\]
the result follows.
\end{proof}

We now have the necessary tools to begin proving the fluid limit result.

Fix $R > e$ and let $l(n,R,d) = R^{-1} + d/n$ and
\[
\tilde{\mathcal{S}}^{n,d} = \left\{\xi \in \mathcal{S}^{n,d}: \xi_1 \geq
  l(n,R,d), \quad \sum_{i=2}^{d+1} \xi_i \leq R \right\}.
\]
Let
\begin{align*}
T^{R,d,1}_n & = \inf\left\{t \geq 0: \bar{X}_1^n(t) < l(n,R,d) \right\}, \\
T^{R,d,2}_n & = \inf\left\{t \geq 0: \bar{Y}_2^n(t) > R \right\}
\end{align*}
and set $T^{R,d}_n = T_n^{R,d,1} \wedge T_n^{R,d,2}$.

\begin{lemma} \label{lem:drift}
For $\xi \in \tilde{\mathcal{S}}^{n,d}$,
there exists a constant $C(R)$, depending only on $R$, such that
\[
\| \bar{\beta}^{n,d}(\xi) - b^{(d)}(\xi) \| \leq \frac{C(R)}{\log
n}.
\]
It follows that for $t_0 < \infty$,
\[
\int_0^{T_n^{R,d} \wedge t_0} \|\bar{\beta}^{n,d}(\bar{X}^{n,d}(t))
- b^{(d)}(\bar{X}^{n,d}(t))\| dt \leq \frac{C(R)t_0}{\log n}.
\]
\end{lemma}

\begin{proof}
We must perform some elementary (but rather involved) calculations.
From the rates of the process we will calculate the co-ordinates of
$\beta^{n,d}(m)$ in turn. Recall first that if $M$ active blocks
are present in the partition, the next event involves the
coalescence of precisely $j$ of them at rate
$\ch{M}{j}\lambda_{M,j}=\frac{M}{j(j-1)}$. Thus, we have
\begin{align*}
\beta_1^{n,d}(m) & = -\rho m_1- \sum_{j=2}^{M}  \frac{M}{j(j-1)}
\sum_{b_1=1}^{m_1} b_1 \frac{\ch{m_1}{b_1} \ch{M-m_1}{j-b_1}}
{\ch{M}{j}}
\\
& = -\rho m_1 - m_1 h\left(M\right).
\end{align*}
For $2 \leq k \leq d$,
\begin{align*}
\beta_{k}^{n,d}(m) & = -\rho m_k  
- \sum_{j=2}^{M} \frac{M}{j(j-1)} \sum_{b_k=1}^{m_k} b_k \frac{\ch{m_k}{b_k}
\ch{M-m_k}{j-b_k}} {\ch{M}{j}} \\
& \hspace{1.7cm} + \sum_{j=2}^{k} \frac{M}{j(j-1)}\!\!\!
\sumstack{0 \leq b_1, b_2, \ldots, b_{k-1} \leq j}{\sum_{l=1}^{k-1}
lb_l = k, \sum_{l=1}^{k-1} b_l = j}\!\!\! \frac{ \ch{m_1}{b_1}
\cdots \ch{m_{k-1}}{b_{k-1}} }{ \ch{M}{j} } \\
& = -\rho m_k - m_k h\left(M\right) + \sum_{j=2}^{k}
\frac{M}{j(j-1)} \!\!\!\sumstack{0 \leq b_1, b_2, \ldots, b_{k-1}
\leq j}{\sum_{l=1}^{k-1} lb_l = k, \sum_{l=1}^{k-1} b_l = j}\!\!\!
\frac{ \ch{m_1}{b_1} \cdots \ch{m_{k-1}}{b_{k-1}} } { \ch{M}{j} }.
\end{align*}
For the $(d+1)$th co-ordinate we have
\begin{align*}
\beta_{d+1}^{n,d}(m) & =  -\rho m_{d+1}  - \sum_{j=2}^{M}
\frac{M}{j(j-1)} \sum_{b_{d+1}=1}^{m_{d+1}} \negmedspace (b_{d+1}
- 1) \frac{ \ch{m_{d+1}}{b_{d+1}}\ch{M-m_{d+1}}{j-b_{d+1}}}
{\ch{M}{j}} \\
& \hspace*{2.1cm}+ \sum_{j=2}^{M} \frac{M}{j(j-1)}\!\!\!
\sumstack{0 \leq b_1, b_2, \ldots, b_{d} \leq j}{\sum_{l=1}^{d} lb_l
\geq d+1, \sum_{l=1}^{d} b_l = j}\!\!\! \frac{ \ch{m_1}{b_1} \cdots
\ch{m_{d}}{b_{d}} }
{ \ch{M}{j} } \\
& = -\rho m_{d+1} - m_{d+1} h\left(M\right)  + \sum_{j=2}^{M}
\frac{M}{j(j-1)} \!\!\!\sumstack{0 \leq b_1, b_2, \ldots, b_{d+1}
\leq j}{\sum_{l=1}^{d+1} lb_l \geq d+1, \sum_{l=1}^{d+1} b_l =
j}\!\!\! \frac{ \ch{m_1}{b_1} \cdots \ch{m_{d+1}}{b_{d+1}} } {
\ch{M}{j} }.
\end{align*}
Finally,
\[
\beta_{d+2}^{n,d}(m)  = \rho m_d.
\]

Using (\ref{eqn:betabar}) and the notation $m=(m_1,\ldots,m_{d+2})$,
we obtain the following expressions:
\begin{align*}
\bar{\beta}_1^{n,d}(\xi) &= -\frac{\rho}{\log n} \xi_1 - \frac{\xi_1
h\left(M\right)}{\log
n}, \\
\intertext{for $2 \leq k \leq d$,}
\bar{\beta}_k^{n,d}(\xi) &= -\frac{\rho}{\log n} \xi_k - \frac{\xi_k
h\left(M\right)}{\log n}+\frac{1}{n}\sum_{j=2}^{k}
\frac{M}{j(j-1)}  \!\!\!\sumstack{0 \leq b_1, b_2, \ldots, b_{k-1}
\leq j}{\sum_{l=1}^{k-1} lb_l = k, \sum_{l=1}^{k-1} b_l = j}
\!\!\!\frac{ \ch{m_1}{b_1} \cdots \ch{m_{k-1}}{b_{k-1}} } { \ch{M}{j}}, \\
\bar{\beta}_{d+1}^{n,d}(\xi) &= -\frac{\rho}{\log n} \xi_{d+1} -
\frac{\xi_{d+1} h\left(M\right)}{\log
n}+\frac{1}{n}\sum_{j=2}^{M} \frac{M}{j(j-1)} \!\!\!\sumstack{0
\leq b_1, b_2, \ldots, b_{k-1} \leq j}{\sum_{l=1}^{d+1} lb_l \ge
d+1, \sum_{l=1}^{d+1} b_l = j}\!\!\! \frac{
\ch{m_1}{b_1} \cdots \ch{m_{d+1}}{b_{d+1}} } { \ch{M}{j} },\\
\bar{\beta}_{d+2}^{n,d}(\xi) & = \rho \xi_d.
\end{align*}

Bearing in mind that $M=n\left(\xi_1+\frac{1}{\log n}
   \sum_{i=2}^{d+1}\xi_i\right)$, and using Lemma~\ref{lem:htolog}, we get
\begin{equation*}
|\bar{\beta}_1^{n,d}(\xi) - b^{(d)}_1(\xi)| \leq \frac{(\rho + \log
R)}{\log n} \xi_1.
\end{equation*}
Consider now the sum in the  expression for
$\bar{\beta}_k^{n,d}(\xi)$ when $2 \leq k \leq d$.  We split it into
two parts, $j=k$ and $2 \leq j \leq k-1$.  The $j=k$ term is
\[
\frac{\xi_1 + \frac{1}{\log n} \sum_{i=2}^{d+1} \xi_i}{k(k-1)}
\frac{\ch{n \xi_1}{k}}{\ch{n \xi_1 + \frac{n}{\log n} \sum_{i=2}^{d+1}
\xi_i}{k}}.
\]
By Lemma~\ref{lem:binomcoeff} we have
\[
\left| \frac{\xi_1 + \frac{1}{\log n} \sum_{i=2}^{d+1} \xi_i}{k(k-1)}
\frac{\ch{n \xi_1}{k}}{\ch{n \xi_1 + \frac{n}{\log n} \sum_{i=2}^{d+1}
\xi_i}{k}} - \frac{1}{k(k-1)} \xi_1 \right| \leq \frac{1}{\log n}
\left(1 + \frac{\xi_1}{\xi_1 - d/n}\right) \sum_{i=2}^{d+1} \xi_i.
\]
Turning now to the other term, if $2 \leq j \leq k-1$, we have
\[
\sumstack{0 \leq b_1, b_2, \ldots, b_{k-1} \leq j}{\sum_{l=1}^{k-1}
lb_l = k, \sum_{l=1}^{k-1} b_l = j} \!\!\!\frac{ \ch{m_1}{b_1}
\cdots \ch{m_{k-1}}{b_{k-1}} }
{ \ch{M}{j} } \\
\leq 1 - \frac{\ch{m_1}{j}}{\ch{M}{j}} \leq \frac{j}{\log n}
\frac{\sum_{i=2}^{d+1}\xi_i}{(\xi_1 - d/n)}
\]
and so
\begin{equation*}
\frac{1}{n}\sum_{j=2}^{k-1} \frac{M}{j(j-1)} \!\!\!\!\sumstack{0
\leq b_1, b_2, \ldots, b_{k-1} \leq j}{\sum_{l=1}^{k-1} lb_l = k,
\sum_{l=1}^{k-1} b_l = j} \!\!\!\!\frac{ \ch{m_1}{b_1} \cdots
\ch{m_{k-1}}{b_{k-1}} } { \ch{M}{j} } \leq \frac{1}{\log n}
\left(\xi_1 + \frac{1}{\log n} \sum_{i=2}^{d+1} \xi_i\right)
\frac{\sum_{i=2}^{d+1}\xi_i}{\xi_1 - d/n} h(d).
\end{equation*}
With another application of Lemma~\ref{lem:htolog}, it follows that
\begin{align*}
& |\bar{\beta}_k^{n,d}(\xi) - b^{(d)}_k(\xi)| \\
& \leq \frac{1}{\log n} \left( (\rho +
 \log R)\xi_k + \left(1 + \frac{\xi_1}{\xi_1 - d/n} \right)
\sum_{i=2}^{d+1} \xi_i + \left( \xi_1 + \frac{1}{\log n} \sum_{i=2}^{d+1}
\xi_i \right) \frac{\sum_{i=2}^{d+1} \xi_i}{\xi_1 - d/n} h(d) \right).
\end{align*}
We turn finally to the expression for $\bar{\beta}_{d+1}^{n}(\xi)$.
Consider the sum which constitutes the third term.  We have
\begin{align*}
\sumstack{0 \leq b_1, b_2, \ldots, b_{d+1} \leq j}{\sum_{l=1}^{d+1}
lb_l \geq d+1, \sum_{l=1}^{d+1} b_l = j}\!\!\! \frac{ \ch{m_1}{b_1}
\cdots \ch{m_{d+1}}{b_{d+1}} } { \ch{M}{j} } & = 1 -
\!\!\!\sumstack{0 \leq b_1, b_2, \ldots, b_{d+1} \leq
j}{\sum_{l=1}^{d+1} l b_l \leq d, \sum_{l=1}^{d+1} b_l = j}\!\!\!
\frac{ \ch{m_1}{b_1} \cdots \ch{m_{d+1}}{b_{d+1}} }
{ \ch{M}{j} } \\
& = 1 - \sum_{k=2}^d \negmedspace \sumstack{0 \leq b_1, b_2, \ldots,
b_{d+1} \leq j}{\sum_{l=1}^{d+1} l b_l = k, \sum_{l=1}^{d+1} b_l =
j} \!\!\!\frac{ \ch{m_1}{b_1} \cdots \ch{m_{d+1}}{b_{d+1}} } {
\ch{M}{j} }.
\end{align*}
But then
\begin{align*}
& \frac{1}{n}\sum_{j=2}^{M} \frac{M}{j(j-1)} \!\!\!\sumstack{0
\leq b_1, b_2, \ldots, b_{d+1} \leq j} {\sum_{l=1}^{d+1} lb_l \geq
d+1, \sum_{l=1}^{d+1} b_l = j} \!\!\!\frac{ \ch{m_1}{b_1} \cdots
\ch{m_{d+1}}{b_{d+1}} }
{ \ch{M}{j} } \\
&  =\frac{1}{n}\left( M-1 - \sum_{k=2}^d \frac{M}{k(k-1)}
\frac{\ch{m_1}{k}}{\ch{M}{k}}  - \sum_{k=2}^d \sum_{j=2}^{k-1}
\frac{M}{j(j-1)}\!\!\!\! \sumstack{0 \leq b_1, b_2, \ldots,
b_{k-1} \leq j}{\sum_{l=1}^{k-1} lb_l = k, \sum_{l=1}^{k-1} b_l =
j}\!\!\!\! \frac{ \ch{m_1}{b_1} \cdots \ch{m_{k-1}}{b_{k-1}} } {
\ch{M}{j} }\right)
\end{align*}
and so, arguing as before, we obtain
\begin{align*}
 |\bar{\beta}_{d+1}^{n,d}(\xi) - b^{(d)}_{d+1}(\xi)| & \leq \frac{1}{n} + \frac{1}{\log n} \left( (\rho + \log R)
\xi_{d+1} + \left(\xi_1 + \frac{1}{\log n} \sum_{i=2}^{d+1}
\xi_i \right) \frac{\sum_{i=2}^d \xi_i}{\xi_1 - d/n} dh(d) \right. \\
& \hspace{7.2cm} \left. + \ d \left(1 + \frac{\xi_1}{\xi_1 - d/n} \right)
\sum_{i=2}^{d+1} \xi_i \right).
\end{align*}
It is clear that
\[
|\bar{\beta}_{d+2}^{n,d}(\xi) - b^{(d)}_{d+2}(\xi)| = 0.
\]
Putting everything together, we obtain that
\[
\|\bar{\beta}^{n,d}(\xi) - b^{(d)}(\xi) \| \leq \frac{C(R)}{\log n},
\]
for some constant $C(R)$, whenever $\xi \in
\tilde{\mathcal{S}}^{n,d}$. The final deduction follows easily.
\end{proof}

\begin{lemma} \label{lem:var} Fix $R > e$.  Then there exists a
  constant $C'(R)$, depending only on $R$, such that for $\xi \in
  \tilde{\mathcal{S}}^{n,d}$,
\[
\bar{\alpha}^{n,d}(\xi) \leq \frac{C'(R)}{\log n}.
\]
It follows that for $t_0 < \infty$,
\[
\int_0^{T^{R,d}_n \wedge t_0} \bar{\alpha}^{n,d}(X_t) dt \leq
\frac{C'(R) t_0}{\log n}.
\]
\end{lemma}

\begin{proof}
Recall that for $1 \leq k \leq d+2$ we have
\[
\alpha_k^{n,d}(m) = \sum_{m' \neq m} |m'_k - m_k|^2 q^{n,d}(m,m'),
\]
so that
\[
\alpha^{n,d}(m) = \sum_{k=1}^{d+2} \alpha_k^{n,d}(m).
\]
We will deal with the co-ordinates in turn.
\begin{align*}
\alpha_1^{n,d}(m) & = \rho m_1 + \sum_{j=2}^{M} \frac{M}{j(j-1)}
\sum_{b_1=1}^{m_1} b_1^2 \frac{\ch{m_1}{b_1}
\ch{M-m_1}{j-b_1}} {\ch{M}{j}}  \\
& = \rho m_1 + m_1 (m_1 - 1) + m_1 h(M).
\end{align*}
Hence,
\[
\bar{\alpha}_1^{n,d}(\xi)=\frac{1}{n^2 \log n}
\alpha_1^{n,d}(m)\leq \frac{\xi_1^2}{\log n} + \frac{C_1(R)}{n}
\]
for some constant $C_1(R)$.  For $2 \leq k \leq d$,
\begin{align*}
\alpha_k^{n,d}(m) & =  \rho m_k + \sum_{j=2}^{M}
\frac{M}{j(j-1)} \sum_{b_k=1}^{m_k} b_k^2 \frac{\ch{m_k}{b_k}
\ch{M-m_k}{j-b_k}} {\ch{M}{j}}  \\
& \hspace*{1.4cm} + \sum_{j=2}^{k} \frac{M}{j(j-1)}\!\!\!\!
\sumstack{0 \leq b_1, b_2, \ldots, b_{k-1} \leq j}{\sum_{l=1}^{k-1}
lb_l = k, \sum_{l=1}^{k-1} b_l = j}\!\!\!\! \frac{ \ch{m_1}{b_1}
\cdots
\ch{m_{k-1}}{b_{k-1}} } { \ch{M}{j} } \\
& = \rho m_k + m_k (m_k - 1) + m_k h(M)+   \sum_{j=2}^{k}
\frac{M}{j(j-1)}\!\!\!\! \sumstack{0 \leq b_1, b_2, \ldots,
b_{k-1} \leq j}{\sum_{l=1}^{k-1} lb_l = k, \sum_{l=1}^{k-1} b_l =
j}\!\!\!\! \frac{ \ch{m_1}{b_1} \cdots
\ch{m_{k-1}}{b_{k-1}} } { \ch{M}{j} }.
\end{align*}
Hence,
\[
\bar{\alpha}_k^{n,d}(\xi)=\frac{\log n}{n^2 }
\alpha_k^{n,d}(m)\leq \frac{\xi_k^2}{\log n} + \frac{C_k(R) \log
n}{n},
\]
for some constant $C_k(R)$.  Furthermore,
\begin{align*}
\alpha_{d+1}^{n,d}(m) & = \rho m_{d+1} + \sum_{j=2}^{M}
\frac{M}{j(j-1)} \sum_{b_{d+1}=1}^{m_{d+1}} (b_{d+1} - 1)^2
\frac{\ch{m_{d+1}}{b_{d+1}}\ch{M-m_{d+1}}{j-b_{d+1}}} {\ch{M}{j}} \\
& \hspace*{1.8cm} + \sum_{j=2}^{M} \frac{M}{j(j-1)}\!\!\!\!
\sumstack{0 \leq b_1, b_2, \ldots, b_{d} \leq j}{\sum_{l=1}^{d} lb_l
\geq d+1, \sum_{l=1}^{d} b_l = j} \!\!\!\!\frac{ \ch{m_1}{b_1}
\cdots
\ch{m_{d}}{b_{d}} } { \ch{M}{j} } \\
& \leq \rho m_{d+1} + m_{d+1} (m_{d+1} - 1) + m_{d+1} h(M) \\
& \hspace{1.8cm} +
\sum_{j=2}^{M} \frac{M}{j(j-1)} \!\!\!\!\sumstack{0 \leq b_1,
b_2, \ldots, b_{d} \leq j}{\sum_{l=1}^{d} lb_l \geq d+1,
\sum_{l=1}^{d} b_l = j}\!\!\!\! \frac{ \ch{m_1}{b_1} \cdots
\ch{m_{d}}{b_{d}} } {
\ch{M}{j} }.
\end{align*}
So we also get
\[
\bar{\alpha}_{d+1}^{n,d}(\xi)=\frac{\log n}{n^2 }
\alpha_{d+1}^{n,d}(m)\leq \frac{\xi_{d+1}^2}{\log n} +
\frac{C_{d+1}(R) \log n}{n},
\]
for some constant $C_{d+1}(R)$.  Finally, we have
\[
\alpha_{d+2}^{n,d}(m) = \rho m_d .
\]
So
\[
\bar{\alpha}_{d+2}^{n,d}(\xi)=\frac{\rho \xi_d\log n}{n} \text{ if
$d=1$ and } \bar{\alpha}_{d+2}^{n,d}(\xi)=\frac{\rho \xi_d(\log
n)^2}{n} \text{ if $d \ge 2$}.
\]
Hence,
\[
\bar{\alpha}^{n,d}(\xi) \leq \frac{\|\xi\|^2}{\log n} + \frac{D(R)
(\log n)^2}{n},
\]
for some constant $D(R)$.  Since $\|\xi\|^2 \leq (1+ R)^2$ in
$\tilde{\mathcal{S}}^{n,d}$, we obtain
\[
\bar{\alpha}^{n,d}(\xi) \leq \frac{C'(R)}{\log n},
\]
for some constant $C'(R)$ depending on $R$.
\end{proof}

\subsection*{Proof of Proposition~\ref{prop:fluidlimit}}
We will, in fact, prove the stronger result that, for any $d \geq 1$
and any $0 <\delta < 1$,
\begin{equation} \label{eqn:desired}
\Prob{\sup_{0 \leq t \leq t_0} \| \bar{X}^{n,d}(t) - x^{(d)}(t) \| \geq
  (\log n)^{\frac{\delta-1}{2}}} \to 0
\end{equation}
as $n \to \infty$.  Fix $d \geq 1$.  We follow the method used in Theorem 3.1
of Darling and Norris~\cite{Darling/Norris} and start by noting that
$\bar{X}^{n,d}(t)$ has the following standard decomposition
\begin{equation}
\bar{X}^{n,d}(t) = \bar{X}^{n,d}(0) + M^{n,d}(t) + \int_0^t
\bar{\beta}^{n,d}(\bar{X}^{n,d}(s)) ds, \label{eqn:Doob}
\end{equation}
where $(M^{n,d}(t))_{t \geq 0}$ is a martingale in the natural filtration
of $\bar{X}^{n,d}$.  Since
\[
x^{(d)}(t) = x^{(d)}(0) + \int_0^t b^{(d)}(x^{(d)}(s)) ds
\]
and $\bar{X}^{n,d}(0) = x^{(d)}(0)$ for all $n \in \N$, we have
\begin{align}
  & \sup_{0 \leq s \leq t} \|\bar{X}^{n,d}(s) - x^{(d)}(s) \| \notag \\
  & \qquad \leq \sup_{0 \leq s \leq t} \| M^{n,d}(s) \| + \int_0^t
  \|\bar{\beta}^{n,d}(\bar{X}^{n,d}(s)) - b^{(d)}(\bar{X}^{n,d}(s)) \| ds
  \notag \\
& \qquad \quad + \int_0^t
  \|b^{(d)}(\bar{X}^{n,d}(s)) - b^{(d)}(x^{(d)}(s)) \| ds. \label{eqn:decomp}
\end{align}

Recall that $K$ is the Lipschitz constant of $b^{(d)}$.  Fix $R > e$ and let
\begin{align*}
\Omega_{n,d,1} & = \left\{ \int_0^{T_n^{R,d} \wedge t_0}
  \|\bar{\beta}^{n,d}(\bar{X}^{n,d}(t)) - b^{(d)}(\bar{X}^{n,d}(t))\| dt
\leq \frac{1}{2} (\log n)^{\frac{\delta - 1}{2}}e^{-Kt_0} \right\}, \\
\Omega_{n,d,2} & = \left\{ \sup_{0 \leq t \leq T_n^{R,d} \wedge t_0}
  \|M^{n,d}(t)\| \leq  \frac{1}{2} (\log n)^{\frac{\delta -
      1}{2}}e^{-Kt_0}\right\} \\
\intertext{and}
\Omega_{n,d,3} & = \left\{ \int_0^{T_n^{R,d} \wedge t_0}
  \bar{\alpha}^{n,d}(\bar{X}^{n,d}(t)) dt \leq \frac{C'(R) t_0}{\log n} \right\}.
\end{align*}
From (\ref{eqn:decomp}) we obtain that for $t < T_n^{R,d} \wedge t_0$ and
on the event $\Omega_{n,d,1} \cap \Omega_{n,d,2}$,
\[
\sup_{0 \leq s \leq t} \|\bar{X}^{n,d}(s) - x^{(d)}(s) \|
\leq (\log n)^{\frac{\delta -1}{2}} e^{-Kt_0}
+ K \int_0^t \sup_{0 \leq r \leq s} \|\bar{X}^{n,d}(r) - x^{(d)}(r) \| ds.
\]
Hence, by Gronwall's lemma,
\[
\sup_{0 \leq t \leq T_n^{R,d} \wedge t_0} \|\bar{X}^{n,d}(t) - x^{(d)}(t) \|
\leq (\log n)^{\frac{\delta -1}{2}}.
\]
Now, by Doob's $L^2$-inequality,
\[
\E{\sup_{0 \leq t \leq T_n^{R,d} \wedge t_0} \| M^{n,d}(t) \|^2}
\leq 4 \E{\| M^{n,d}(T_n^{R,d} \wedge t_0) \|^2} \leq 4
\E{\int_0^{T_n^{R,d} \wedge t_0}
\bar{\alpha}^{n,d}(\bar{X}^{n,d}(s)) ds}.
\]
Combined with Chebyshev's inequality, this tells us that
\[
\Prob{\sup_{0 \leq t \leq T_n^{R,d} \wedge t_0} \| M^{n,d}(t) \|
\geq \frac{1}{2}(\log n)^{\frac{\delta- 1}{2}} e^{-Kt_0}, \Omega_{n,d,3}}
\leq \frac{16 C'(R) t_0e^{2Kt_0}}{(\log n)^{\delta}}.
\]
Hence, $\Prob{\Omega_{n,d,2} \setminus \Omega_{n,d,3}}
\to 0$.  By Lemmas~\ref{lem:drift} and \ref{lem:var}, we have
$\Prob{\Omega_{n,d,1}} \to 1$ and $\Prob{\Omega_{n,d,3}} \to 1$ as $n \to
\infty$.  But
\[
\Prob{\sup_{0 \leq t \leq T_n^{R,d} \wedge t_0} \|\bar{X}^{n,d}(t) - x^{(d)}(t) \|
> (\log n)^{\frac{\delta -1}{2}}} \leq \frac{16 C'(R)
  t_0e^{2Kt_0}}{(\log n)^{\delta}} + \Prob{\Omega_{n,d,1}^{\mathrm{c}}
  \cup \Omega_{n,d,3}^{\mathrm{c}}},
\]
which clearly tends to 0 as $n \to \infty$.

In fact, we wish to prove this result for $t_0$ rather than $T_n^{R,d}
\wedge t_0$.  Set
\[
\Omega_{n,R,d} = \left\{\sup_{0 \leq t \leq T_n^{R,d} \wedge t_0} \|\bar{X}^{n,d}(t) - x^{(d)}(t) \|
\leq (\log n)^{\frac{\delta -1}{2}}\right\}.
\]
Since $T_n^{R,d} = T_n^{R,d,1} \wedge T_n^{R,d,2}$, it will suffice for us to
show that $T_n^{R,d,1} > t_0$ and $T_n^{R,d,2} > t_0$ on $\Omega_{n,R,d}$
for all large enough $n$ and $R$.

Note firstly that $x_1(t) > 0$ for all $t \geq 0$ and that $x_1(t)$
decreases to $0$ as $t \to \infty$.  Take $n$ and $R$ large enough that
$x_1(t_0) > (\log n)^{\frac{\delta -1}{2}} + l(n,R,d)$.  Then
on $\Omega_{n,R,d}$,
\begin{align*}
\inf_{0 \leq t \leq T_n^{R,d,1} \wedge t_0} |\bar{X}^n_1(t)|
& \geq \inf_{0 \leq t \leq T_n^{R,d,1} \wedge t_0} \biggl| x_1(t) -
|\bar{X}^n_1(t) - x_1(t)| \biggr| \\
& \geq x_1(t_0)
- \sup_{0 \leq t \leq T_n^{R,d} \wedge t_0} \|\bar{X}^{n,d}(t) - x^{(d)}(t) \| \\
& > l(n,R,d).
\end{align*}
Note now that $0 \leq y_{2}(t) \leq e^{-1} < R$ for all $t \geq 0$.
Take $n$ to be sufficiently big that $(\log n)^{\frac{\delta -1}{2}} +
e^{-1} < R$.  Then on $\Omega_{n,R,d}$, we have
\[
\sup_{0 \leq t \leq T_n^{R,d,2} \wedge t_0} |\bar{Y}_2^n(t)|
\leq \sup_{0 \leq t \leq T_n^{R,d} \wedge t_0} \|\bar{X}^{n,d}(t) - x^{(d)}(t) \|
+ \sup_{0 \leq t \leq t_0} |y_2(t)| < R.
\]
The desired result (\ref{eqn:desired}) follows.

\subsection*{Proof of Proposition~\ref{prop:stoptime}}
For convenience we will write
\[
z_1(\infty) \deff \lim_{t \to \infty} z_1(t) = \rho \quad \text{and,
  for $k \geq 2$,} \quad z_k(\infty) \deff \lim_{t \to \infty} z_k(t) =
\frac{\rho}{k(k-1)}.
\]
Since we can take $t_0$ arbitrarily large, we can make $z_d(t_0)$
arbitrarily close to $z_d(\infty)$.  With high probability, on the
time interval $\left[0,t_0\right]$, $\frac{\log n}{n}
Z_1^n(\frac{t}{\log n})$ stays close to $z_1(t)$ and, likewise, for
$d \geq 2$, $\frac{(\log n)^2}{n} Z_d^n(\frac{t}{\log n})$ stays
close to $z_d(t)$.  So the work of this proof will be to demonstrate
that $Z^n_d(t)$ does not do anything ``nasty'' between times
$\frac{t_0}{\log n}$ and $T_n$.  (Note that this interval is
potentially quite long: absorption for the coalescent takes place at
about time $\log \log n$; see Proposition 3.4 of
\cite{Goldschmidt/Martin}.)  We will have to split the interval
$\left[\frac{t_0}{\log n}, T_n\right]$ into two parts and deal with
process separately on each.

The statement of Proposition~\ref{prop:stoptime} fixes some $\delta >
0$.  Take also $\eta > 0$ and fix $t_0$ such that
\begin{gather*}
2 x_1(t_0) +  y_2(t_0) < \frac{\delta \eta}{24 \rho} \\
\text{and, for $k \geq 1$,} \quad
|z_k(t_0) - z_k(\infty)| < \frac{\delta}{3}.
\end{gather*}
(We can do this uniformly in $k$ because of the special form of the
functions $x_1(t)$, $y_2(t)$ and $z_k(t), k \geq 1$.)  Take $0 <
\epsilon < \frac{\delta}{3} \wedge \frac{\delta \eta}{72\rho} \wedge
x_1(t_0)$.  Let
\[
\Omega_{n,d} = \left\{ \sup_{0 \leq t \leq t_0} \|\bar{X}^{n,d}(t) - x^{(d)}(t)\| <
\epsilon \right\}.
\]
Since $(\log n) T_n^{R,d,1} \leq T_n$ and $\epsilon < x_1(t_0)$, we
know by the argument in the proof of Proposition~\ref{prop:fluidlimit}
that $T_n > \frac{t_0}{\log n}$ on $\Omega_{n,1}$ for all sufficiently
large $n$ and $R$.  Let
\[
\tau_n = \inf\left\{t \geq \frac{t_0}{\log n}: X^n_1(t) <
  \frac{n}{(\log n)^3} \text{ and } Y^n_2(t) < \frac{n}{(\log n)^3}\right\}.
\]
We will first deal with the time interval $\left[\frac{t_0}{\log n},
  \tau_n \right)$.

\begin{lemma}
For sufficiently large $n$, we have
\begin{align}
\frac{\log n}{n} \E{\left(Z_1^n(\tau_n) - Z_1^n\left(\frac{t_0}{\log
      n}\right)\right)\IO} & \leq \frac{\delta \eta}{6}
\label{eqn:firstint1} \\
\intertext{and, for $k \geq 2$,}
\frac{(\log n)^2}{n} \E{\left(Z_k^n(\tau_n) - Z_k^n\left(\frac{t_0}{\log
      n}\right)\right)\IO} & \leq \frac{\delta \eta}{6}.
\label{eqn:firstint2}
\end{align}
\end{lemma}

\begin{proof}
Consider a new process $(\tilde{X}_1^n(t), \tilde{Y}_2^n(t))_{t \geq
  0}$ which starts from $(\bar{X}_1^n(t_0), \bar{Y}^n_2(t_0))$ and has
the same dynamics as $(\bar{X}_1^n(t), \bar{Y}_2^n(t))_{t \geq 0}$,
except with a stochastic time-change which means that time is now run
at instantaneous rate $(\tilde{X}^n_1(t) + \tilde{Y}^n_2(t))^{-1}$.  In
other words, if
\[
U^n(s) = \int_{0}^{s} \left[\frac{1}{n}X^n_1\left(\frac{t_0+u}{\log n}\right)
+ \frac{\log n}{n}Y^n_2\left(\frac{t_0+u}{\log n}\right)\right] du
\]
and
\[
V^n(t) = \inf\left\{s \geq 0: U^n(s) \geq t \right\}
\]
then
\[
\tilde{X}^n_1(t) = \frac{1}{n} X^n_1\left(\frac{t_0 + V^n(t)}{\log n}\right),
\quad \tilde{Y}^n_2(t)= \frac{\log n}{n} Y_2^n\left(\frac{t_0 +
    V^n(t)}{\log n}\right).
\]
Let $\tilde{\tau}_n = U^n((\log n)\tau_n - t_0) = \inf\left\{t \geq 0:
  \tilde{X}^n_1(t) < \frac{1}{(\log n)^3} \text{ and }
  \tilde{Y}^n_2(t) < \frac{1}{(\log n)^2} \right\}$.  Then we have
$\tilde{X}_1^n(\tilde{\tau}_n) = \frac{1}{n} X_1^n(\tau_n)$
and $\tilde{Y_2^n}(\tilde{\tau}_n) = \frac{\log n}{n}
Y_2^n(\tau_n)$.

The process $(\tilde{X}_1^n(t), \tilde{Y}_2^n(t))_{t \geq 0}$ has
drift vector $\tilde{\beta}^n(\xi)$ in state $\xi$, where
\begin{align*}
\tilde{\beta}^n_1(\xi) &
= -\frac{\rho \xi_1}{(\xi_1 + \xi_2) \log n} - \frac{\xi_1 h(n \xi_1 +
  \frac{n}{\log n} \xi_2)}{(\xi_1 + \xi_2) \log n} \\
\tilde{\beta}^n_2(\xi) &
= -\frac{\rho \xi_2}{(\xi_1 + \xi_2) \log n} - \frac{\xi_2 h(n \xi_1 +
  \frac{n}{\log n} \xi_2)}{(\xi_1 + \xi_2) \log n} + \frac{\xi_1 +
  \frac{1}{\log n} \xi_2 - \frac{1}{n}}{(\xi_1 + \xi_2)}.
\end{align*}
Let
\[
A^n(t) = 2 \tilde{X}^n_1(t) + \tilde{Y}^n_2(t) + t, \quad t \geq 0.
\]
Then $A^n(t)$ has drift
\begin{align*}
& 2 \tilde{\beta}_1^n(\xi) + \tilde{\beta}_2^n(\xi) +  1 \\
& \quad = -\frac{2\xi_1 + \xi_2}{\xi_1 + \xi_2} \left( \frac{h(n \xi_1 +
      \frac{n}{\log n} \xi_2)}{\log n} - 1 \right) + \frac{1}{\log n}
  \frac{\xi_2}{\xi_1 + \xi_2} -\frac{\rho (2 \xi_1 + \xi_2)}{(\xi_1 +
    \xi_2) \log n} - \frac{1}{n(\xi_1 + \xi_2)},
\end{align*}
in state $\xi$.  Intuitively, this is small for large $n$ and so
$(A^n(t))_{t \geq 0}$ is almost a martingale.  More
rigorously, we have
\[
2 \tilde{\beta}_1^n(\xi) + \tilde{\beta}_2^n(\xi) +  1 \\
\leq \frac{2 \xi_1 + \xi_2}{\xi_1 + \xi_2} \left( 1 - \frac{h(n \xi_1 +
      \frac{n}{\log n} \xi_2)}{\log n} \right)
+ \frac{1}{\log n} \frac{\xi_2}{\xi_1 + \xi_2}.
\]
Lemma~\ref{lem:htolog} remains true if we replace $R$ by $(\log n)^3$. So,
since $\frac{\xi_1}{\xi_1 + \xi_2}, \frac{\xi_2}{\xi_1 + \xi_2} \leq
1$ in $\mathcal{S}^{n,d}$, we obtain
\[
2 \tilde{\beta}_1^n(\xi) + \tilde{\beta}_2^n(\xi) +  1 \leq
\frac{6 \log \log n + 1}{\log n},
\]
whenever $\xi \in \mathcal{S}^{n,d}$ and $\xi_1 + \frac{1}{\log n}
\xi_2 \in \frac{1}{n} \Z
\cap \left[\frac{1}{(\log n)^3},1\right]$.

By the same standard decomposition as at (\ref{eqn:Doob}), there
exists a zero-mean martingale $(\tilde{M}^n(t))_{t \geq 0}$ such that
\[
A^n(t) = A^n(0) + \tilde{M}^n(t) + \int_0^t (2
\tilde{\beta}_1^n(\tilde{X}^n(s)) + \tilde{\beta}_2^n(\tilde{X}^n(s))
+  1)ds.
\]
Fix $t_1 > 0$.  For any particular $n$, $A^n(t)$ and $\int_0^t (2
\tilde{\beta}_1^n(\tilde{X}^n(s)) + \tilde{\beta}_2^n(\tilde{X}^n(s))
+ 1)ds$ are bounded on the time interval $[0,t_1]$ and so we may apply
the Optional Stopping Theorem to obtain that
\begin{align*}
& \E{\left(2 \tilde{X}_1^n(\tilde{\tau}_n \wedge t_1) +
    \tilde{Y}_2^n(\tilde{\tau}_n \wedge t_1)
  + (\tilde{\tau}_n \wedge t_1) \right) \IO} \\
& \qquad = \E{A^n(\tilde{\tau}_n \wedge t_1) \IO} \\
& \qquad = \E{A^n(0) \IO} + \E{\int_0^{\tilde{\tau}_n \wedge t_1} (2
  \tilde{\beta}_1^n(\tilde{X}^n(t)) +
  \tilde{\beta}_2^n(\tilde{X}^n(t)) +  1)dt \IO} \\
& \qquad = \E{\left(2 \bar{X}_1^n(t_0) + \bar{Y}_2^n(t_0) \right) \IO}
+ \E{\int_0^{\tilde{\tau}_n \wedge t_1} (2
  \tilde{\beta}_1^n(\tilde{X}^n(t)) +
  \tilde{\beta}_2^n(\tilde{X}^n(t)) +  1)dt \IO}.
\end{align*}
We have that $\tilde{X}_1^n(\tilde{\tau}_n \wedge t_1)$ and
$\tilde{Y}_2^n(\tilde{\tau}_n \wedge t_1)$ are both non-negative and so
\begin{align*}
\E{(\tilde{\tau}_n \wedge t_1)\IO}
& \leq \left(1 - \frac{6 \log \log n + 1}{\log
    n}\right)^{-1} \E{\left(2 \bar{X}_1^n(t_0) + \bar{Y}_2^n(t_0)
  \right) \IO} \\
& \leq 2(2x_1(t_0) + y_2(t_0) + 3 \epsilon) \\
& < \frac{\delta \eta}{6 \rho},
\end{align*}
since, for large enough $n$, $\left(1 - \frac{6 \log \log n + 1 }{\log
    n}\right)^{-1}$ is bounded above by $2$ and we have assumed that
$2 x_1(t_0) + y_2(t_0) < \frac{\delta \eta}{24 \rho}$ and $\epsilon <
\frac{\delta \eta}{72 \rho}$.  Letting $t_1 \uparrow \infty$, we obtain
by monotone convergence that
\[
\E{\tilde{\tau}_n \IO} < \frac{\delta \eta}{6 \rho}.
\]

Now, by a further application of the Optional Stopping Theorem and
monotone convergence,
\begin{align*}
\frac{\log n}{n} \E{\left(Z_1^n(\tau_n) - Z_1^n\left(\frac{t_0}{\log
        n}\right)\right) \IO}
& = \frac{\log n}{n} \E{\int_{\frac{t_0}{\log n}}^{\tau_n} \rho
  X_1^n(t) dt \IO} \\
& = \E{\int_0^{\tilde{\tau}_n} \frac{\rho
    \tilde{X}_1^n(s)}{\tilde{X}^n_1(s) + \tilde{Y}^n_2(s)} ds \IO} \\
& \leq \rho \E{\tilde{\tau}_n \IO},
\end{align*}
by changing variable in the integral.  Similarly, for $k \geq 2$,
\begin{align*}
\frac{(\log n)^2}{n} \E{\left(Z_k^n(\tau_n) -
    Z_k^n\left(\frac{t_0}{\log n}\right)\right) \IO}
& = \frac{(\log n)^2}{n} \E{\int_{\frac{t_0}{\log n}}^{\tau_n} \rho
  X_k^n(t) dt \IO} \\
& \leq \frac{(\log n)^2}{n} \E{\int_{\frac{t_0}{\log n}}^{\tau_n} \rho
  Y_2^n(s) ds \IO} \\
& = \E{\int_0^{\tilde{\tau}_n} \frac{ \rho
    \tilde{Y_2^n(s)}}{\tilde{X}^n_1(s) + \tilde{Y}^n_2(s)} ds \IO} \\
& \leq \rho \E{\tilde{\tau}_n \IO}.
\end{align*}
The result follows.
\end{proof}

From (\ref{eqn:firstint1}) and (\ref{eqn:firstint2}) and Markov's inequality,
\begin{align*}
& \Prob{\left| \frac{\log n}{n} \left(Z_1^n(\tau_n) - Z_1^n
      \left(\frac{t_0}{\log n}\right) \right) \right| > \frac{\delta}{3},
  \Omega_{n,1}} \\
& \qquad \leq \frac{3 \log n}{n \delta} \E{\left(Z_1^n(\tau_n) -
  Z_1^n\left(\frac{t_0}{\log n}\right)\right) \IO}
\leq \frac{\eta}{2} \\
\intertext{and, for $k \geq 2$,}
& \Prob{\left| \frac{(\log n)^2}{n} \left(Z_k^n(\tau_n) -
      Z_k^n\left(\frac{t_0}{\log n}\right) \right) \right| >
  \frac{\delta}{3}, \Omega_{n,1}} \\
& \qquad \leq \frac{3 (\log n)^2}{n \delta} \E{\left(Z_k^n(\tau_n) -
  Z_k^n\left(\frac{t_0}{\log n}\right)\right) \IO}
\leq \frac{\eta}{2}.
\end{align*}

Note that we necessarily have $\tau_n \leq T_n$.  Since $Z_k^n(t)$ is
increasing for all $k \geq 1$ and $Z_k^n(T_n) - Z_k^n(\tau_n) \leq
X^n_1(\tau_n) + Y^n_2(\tau_n) < \frac{2n}{(\log n)^3}$ for all $k
\geq 1$, we have that
\[
\frac{\log n}{n} \left(Z_1^n(T_n) - Z_1^n(\tau_n) \right) <
\frac{2}{(\log n)^2}
\]
and, for $k \geq 2$,
\[
\frac{(\log n)^2}{n} \left(Z_k^n(T_n) - Z_k^n(\tau_n) \right) <
\frac{2}{\log n}.
\]
For $n > \exp(\frac{6}{\delta})$ these quantities are both less than
$\frac{\delta}{3}$.  On $\Omega_{n,1}$ we have
\[
\left|\frac{\log n}{n} Z_1^n\left(\frac{t_0}{\log n}\right) -
  z_1(t_0)\right| \leq \frac{\delta}{3}.
\]
By taking $n$ sufficiently large, we have by
Proposition~\ref{prop:fluidlimit} that
$\Prob{\Omega_{n,1}^{\mathrm{c}}} < \frac{\eta}{2}$ and so we conclude
that
\[
\Prob{\left| \frac{\log n}{n} Z_1^n(T_n) - z_1(\infty) \right| >
  \delta} < \eta.
\]

Now consider the case $d \geq 2$.  On $\Omega_{n,d}$ we have
\[
\left|\frac{(\log n)^2}{n} Z_d^n\left(\frac{t_0}{\log n}\right)
- z_d(t_0)\right| \leq \frac{\delta}{3}
\]
and, by taking $n$ sufficiently large, we have by
Proposition~\ref{prop:fluidlimit} that
$\Prob{\Omega_{n,1}^{\mathrm{c}}} + \Prob{\Omega_{n,d}^{\mathrm{c}}} <
\frac{\eta}{2}$.  Hence,
\[
\Prob{\left| \frac{(\log n)^2}{n} Z_d^n(T_n) - z_d(\infty) \right| >
  \delta}  < \eta.
\]
But $\eta$ was arbitrary and so this completes the proof of
Proposition~\ref{prop:stoptime}.

\section{Comments}
\subsection{Asymptotic frequencies}

It would be very interesting to have a better understanding of the
distribution of the asymptotic frequency sequence of the allelic
partition associated with the Bolthausen-Sznitman coalescent. In
\cite{Gnedin/Hansen/Pitman}, Gnedin, Hansen and Pitman obtain
relations between the total number of blocks $N(n)$ of an exchangeable
random partition restricted to the set $\{1,\ldots,n\}$ and the
asymptotic form of the sequence $(f_i^\downarrow)_{i\ge 1}$.  More
precisely, they prove that, for any $\alpha\in (0,1)$ and any function
$\ell : \mathbb{R}_+\rightarrow \mathbb{R}_+$, slowly varying at
infinity, we have
\[
\frac{N(n)}{\Gamma(1-\alpha) n^\alpha
  \ell(n)} \convas 1 \Longleftrightarrow
\frac{\# \{i\ge 1: f_i^{\downarrow} \ge x\}}{\ell(1/x)x^{-\alpha}} \convas 1
\text{ as $x\rightarrow 0+$} \Longleftrightarrow
\frac{f_i^{\downarrow}}{\ell^*(i) i^{-1/\alpha}} \convas 1,
\]
where $\ell^*$ is also a slowly varying function which can be
expressed in term of $\alpha$ and $\ell$.

It would be nice to have a similar result for the allelic partition
associated with the Bolthausen-Sznitman coalescent. There are,
however, two main difficulties: first, we would need almost sure
convergence of the rescaled process $N(\cdot)$, whereas here we have
only established convergence in probability.  Second, the
Bolthausen-Sznitman coalescent corresponds to the critical case
$\alpha=1$ for which the first of the above equivalences no longer
holds.  In this setting, according to Proposition $18$ of Gnedin,
Hansen and Pitman~\cite{Gnedin/Hansen/Pitman}, we have only the
implication:
\[
x(\log x)^2 \# \{i\ge 1: f_i^{\downarrow} \ge x\} \convas \rho \text{
  as $x \to 0+$} \Longrightarrow \frac{\log n}{n} N(n)\convas \rho
\]
and, in addition, that
\[
\frac{\log n}{n} N_1(n) \convas \rho \quad \text{and} \quad 
\frac{(\log n)^2}{n} N_k(n) \convas \frac{\rho}{k(k-1)}, k \geq 2.
\]
The form of the limits is, of course, basically the same as in our
Theorem~\ref{thm:main} and so we might expect to find that
$f_i^{\downarrow} \sim \frac{\rho}{ i (\log i)^2}$ as $i$ tends to
infinity.

\subsection{Beta coalescents}

The fluid limit methods used in this paper can, in principle, be
extended to deal with other classes of coalescent process. For
instance, the method seems to work for the Beta coalescents with
parameter $\alpha\in (1,2)$.  However, the calculations are more
complicated than in the Bolthausen-Sznitman case. Indeed, for the
Bolthausen-Sznitman coalescent, the active partition is mostly
composed of singletons at any time, which essentially enables us to
neglect collisions between non-singleton blocks. This approximation
does not hold for the Beta coalescents with $\alpha \in (1,2)$.  Since
the relevant result has already been proved by Berestycki, Berestycki
and Schweinsberg~\cite{BBS2,BBS1} by other methods, we will not give
the details.

We may also consider the Beta coalescents with parameter
$\alpha\in(0,1)$. M\"ohle's result (\ref{eqn:Moehle}) that the total
number of mutations along the coalescent tree, re-scaled by $n$,
converges in distribution to some non-degenerate random variable
suggests that here we may expect to have convergence in distribution
of the allelic partition to a random vector.  Clearly, the fluid limit
methods used in the present paper do not adapt to this situation, but
we can still use them to investigate the expected value of the number
of blocks of different sizes.  Indeed, the drift of the re-scaled
process
\[
\begin{array}{ll}
\left(\frac{X_1^n(t)}{n},\frac{Y_{d+1}^n(t)}{n^\alpha},\frac{Z_{d}^n(t)}{n}\right)
&\text{ if $d=1$}\\
\left(\frac{X_1^n(t)}{n},\frac{X_2^n(t)}{n^\alpha},\ldots,\frac{X_d^n(t)}{n^\alpha},\frac{Y_{d+1}^n(t)}{n^\alpha},\frac{Z_{d}^n(t)}{n^\alpha}\right)
&\text{ if $d\ge2$}
\end{array}
\]
converges to an explicit function $b^{(d)}$ (but the variance
$\bar{\alpha}^{n,d}$ does not tend to 0). This enables us to
conjecture that
\[
N_1(n)\sim C_1 n \quad \mbox{ and }\quad  N_k(n)\sim C_k n^\alpha
\quad \mbox{ for } k\ge 2,
\]
where $C_1,C_2,\ldots$ are strictly positive random variables.  We
intend to address this problem in a future paper.

\section*{Acknowledgments}

A.-L.\ B.\ would like to thank the Statistical Laboratory at the
University of Cambridge for its kind invitation, which was the starting
point of this paper.  The early part of the work was done while C.\
G.\ held the Stokes Research Fellowship at Pembroke College,
Cambridge.  Pembroke's support is most gratefully acknowledged.  For
the later part, C.\ G.\ was funded by EPSRC Postdoctoral Fellowship
EP/D065755/1.  We would like to thank James Norris for several
extremely helpful discussions.

\bibliography{allelicpartition}

\small

\begin{tabbing}
\hspace{8.6cm} \= \kill \\
Anne-Laure Basdevant,\>
Christina Goldschmidt, \\
Laboratoire de Probabilit\a'es et Mod\a`eles Al\a'eatoires,\>
Department of Statistics, \\
Universit\a'e Pierre et Marie Curie (Paris VI),
\> University of Oxford, \\
Case courier 188,
\> 1 South Parks Road,\\
4, place Jussieu,
\> Oxford \\
75252 Paris Cedex 05
\> OX1 3TG \\
France
\> United Kingdom \\
\texttt{anne-laure.basdevant@ens.fr}\> \texttt{goldschm@stats.ox.ac.uk} \\
\texttt{http://www.proba.jussieu.fr/\~{}abasdeva/}\>
\texttt{http://www.stats.ox.ac.uk/\~{}goldschm/}
\end{tabbing}

\end{document}